\documentclass[12pt, a4paper]{article}
\usepackage{setspace}
\usepackage{geometry}
\usepackage{graphicx}
\usepackage{endfloat}
\usepackage{amssymb,amsmath,bm,mathrsfs,makeidx,amsfonts,multirow,stmaryrd,dsfont,cite,bm,amsthm}
\usepackage{natbib}
\newtheorem{thm}{Theorem}
\newtheorem{prop}{Proposition}
\newtheorem{lem}{Lemma}

\newtheorem{rmk}{Remark}

\newtheorem{defn}{Definition}

\begin{document}

\begin{center}
\Large{Test of  Independence for High-dimensional Random Vectors Based on Block Correlation Matrices}
\end{center}

\begin{center}
Zhigang Bao\\
Division of Mathematical Sciences, Nanyang Technological University\\
 Singapore, 637371\\ 
email:\texttt{zgbao@ntu.edu.sg}\\ \vspace{0.2in}
Jiang Hu\\
School of Mathematics and Statistics, Northeast Normal
University\\ China, 130024 \\ email: \texttt{huj156@nenu.edu.cn}\\  \vspace{0.2in}
Guangming Pan\\
Division of Mathematical Sciences, Nanyang Technological University\\  Singapore, 637371 \\ email: \texttt{gmpan@ntu.edu.sg}\\ \vspace{0.2in}
Wang Zhou\\
Department of Statistics and Applied Probability, National University of Singapore,\\
Singapore, 117546\\ email:\texttt{stazw@nus.edu.sg}
\end{center}

\newpage
%
\begin{center}
\textbf{Abstract}
\end{center}
In this paper, we are concerned with the independence test for $k$ high-dimensional sub-vectors of a normal vector, with fixed positive integer $k$. A natural high-dimensional extension of the classical sample correlation matrix, namely block correlation matrix, is raised for this purpose. We then construct the so-called Schott type statistic as our test statistic, which turns out to be a particular linear spectral statistic of the block correlation matrix. Interestingly, the limiting behavior of the  Schott type statistic can be figured out with the aid of the Free Probability Theory and the Random Matrix Theory. Specifically, we will bring the so-called real second order freeness for Haar distributed orthogonal matrices, derived in \cite{MP2013},  into the framework of this high-dimensional testing problem. Our  test does not  require the sample size to be larger than the total or any partial sum of the dimensions of the $k$ sub-vectors. Simulated results show the effect of the Schott type statistic,  in contrast to  those of the statistics proposed in \cite{JY2013} and \cite{JBZ2013}, is satisfactory.  Real data analysis is also used to illustrate  our method.
\vspace*{.3in}


\newpage

\section{Introduction}
Test of independence for random variables is a very classical hypothesis testing problem, which dates back to the seminal work by \cite{Pearson00}, followed by a huge literature regarding this topic and its variants. One frequently recurring variant is the test of independence for $k$ random vectors, where $k\geq 2$ is an integer. Comprehensive overview and detailed references on this problem can be found
 in most of the textbooks on multivariate statistical analysis. For instance, here we recommend the masterpieces by \cite{Mu1982}  and by \cite{Anderson03I} for more details, in the low-dimensional case. However, due to the increasing demand in the analysis of big data springing up in various fields nowadays, such as genomics, signal processing, microarray, proteomics and finance, the investigation on a high-dimensional extension of this testing problem is much needed, which motivates us to propose a feasible way for it in this work.
%

Let us take a review more specifically on some representative existing results in the literature, after necessary notation is introduced. For simplicity, henceforth, we will use the notation $\llbracket m\rrbracket$ to denote the set $\{1,2,\ldots,m\}$ for any positive integer $m$.
Assume that $\mathbf{x}=(\mathbf{x}_1',\ldots, \mathbf{x}_k')'$ is a $p$-dimensional normal vector, in which $\mathbf{x}_i$ possesses dimension $p_i$ for $i\in\llbracket k\rrbracket$, such that $p_1+\cdots+p_k=p$. Denote by $\boldsymbol{\mu}_i$ the mean vector of the $i$th sub-vector $\mathbf{x}_i$ and by $\Sigma_{ij}$ the cross covariance matrix of $\mathbf{x}_i$ and $\mathbf{x}_j$ for all $i,j\in\llbracket k\rrbracket$. Then $\boldsymbol{\mu}:=(\boldsymbol{\mu}_1',\ldots,\boldsymbol{\mu}_k')'$ and $\Sigma:=(\Sigma_{ij})_{k,k}$ are the mean vector and covariance matrix of $\mathbf{x}$ respectively. In this work, we consider the following hypothesis testing
\begin{eqnarray*}
(\mathbf{T1})\hspace{10ex}\mathbf{H}_0: \Sigma_{ij}=0, \quad i\neq j \hspace{5ex} \text{v.s.} \hspace{5ex} \mathbf{H}_1: \text{ not } \mathbf{H}_0.
\end{eqnarray*}
To this end, we draw $n$ observations of $\mathbf{x}$, namely $\mathbf{x}(1),\ldots,\mathbf{x}(n)$.  In addition, the $i$th sub-vector of $\mathbf{x}(j)$ will be denoted by $\mathbf{x}_i(j)$ for all $i\in \llbracket k\rrbracket$ and $j\in\llbracket n\rrbracket$. Hence, $\mathbf{x}_i(j),j\in\llbracket n\rrbracket$ are $n$ independent observations of $\mathbf{x}_i$. Conventionally, the corresponding sample means will be written as $\bar{\mathbf{x}}=n^{-1}\sum_{j=1}^n\mathbf{x}(j)$ and $\bar{\mathbf{x}}_i=n^{-1}\sum_{j=1}^n\mathbf{x}_i(j)$. With the observations at hand, we can construct the sample covariance matrix  as usual. Set
\begin{eqnarray}
\mathbf{X}:=[\mathbf{x}(1)-\bar{\mathbf{x}},\cdots, \mathbf{x}(n)-\bar{\mathbf{x}}],\quad \mathbf{X}_i:=[\mathbf{x}_i(1)-\bar{\mathbf{x}}_i,\cdots, \mathbf{x}_i(n)-\bar{\mathbf{x}}_i], \quad i\in\llbracket k\rrbracket. \label{3193}
\end{eqnarray}
The sample covariance matrix of $\mathbf{x}$ and the cross sample covariance matrix of $\mathbf{x}_{m}$ and $\mathbf{x}_{\ell}$ will be denoted by $\hat{\Sigma}$ and $\hat{\Sigma}_{m\ell}$ respectively, to wit,
\begin{eqnarray*}
\hat{\Sigma}:=\frac{1}{n-1}\mathbf{X}\mathbf{X}' ,\hspace{8ex} \hat{\Sigma}_{m\ell}:=\frac{1}{n-1}\mathbf{X}_m\mathbf{X}_\ell'.
\end{eqnarray*}

In the classical {\emph{large $n$ and fixed $p$}} case, the likelihood ratio statistic $\Lambda_n:=W_n^{n/2}$  with
\begin{eqnarray*}
W_n:=\frac{|\hat{\Sigma}|}{\prod_{i=1}^k|\hat{\Sigma}_{ii}|}
\end{eqnarray*}
is a favorable one aiming at the testing problem ({\bf{T1}}).
A celebrated limiting law on $\Lambda_n$ under $\mathbf{H}_0$ is
\begin{eqnarray}
-2\kappa\log \Lambda_n\Longrightarrow\chi_\rho^2,\quad \text{as  }n\to\infty  \label{3191}
\end{eqnarray}
where
\begin{eqnarray*}
\kappa=1-\frac{2(p^3-\sum_{i=1}^k p_i^3)+9(p^2-\sum_{i=1}^k p_i^2)}{6n(p^2-\sum_{i=1}^k p_i^2)},\quad
\rho=\frac12(p^2-\sum_{i=1}^kp_i^2).
\end{eqnarray*}
One can refer to \cite{Wilks1935} or Theorem 11.2.5 of \cite{Mu1982}, for instance. The left hand side of (\ref{3191}) is known as the Wilks statistic.

Now, we turn to the high-dimensional setting of interest. A commonly used assumption on dimensionality and sample size in the Random Matrix Theory ({\bf{RMT}}) is that $p$ is proportionally large as $n$, i.e.
\begin{eqnarray*}
p:=p(n),\quad \frac{p}{n}\to y\in(0,\infty),\text{ as } n\to \infty.
\end{eqnarray*}
To employ the existing {\bf{RMT}} apparatus in the sequel, hereafter, we will always work with the above {\emph{large $n$ and proportionally large $p$}} setting. This time, resorting to the likelihood ratio statistic in (\ref{3191}) directly is obviously infeasible,  since  the limiting law (\ref{3191}) is invalid when $p$ tends to infinity along with $n$. Actually, under this setting, the likelihood ratio statistic can still be employed if an appropriate renormalization is performed priori.  In \cite{JY2013}, the authors renormalized the likelihood ratio statistic, and derived its limiting law under $\mathbf{H}_0$ as a central limit theorem (CLT), under the restriction of
\begin{eqnarray}
n>p+1,\quad p_i/n\to y_i\in(0,1), i\in \llbracket k\rrbracket,  \text{ as } n\to \infty. \label{052206}
\end{eqnarray}
One can refer to Theorem 2 in \cite{JY2013} for the details of the normalization and the CLT. One similar result was also obtained in \cite{JBZ2013}, see Theorem 4.1 therein.
However, the condition (\ref{052206}) is indispensable for the likelihood ratio statistic which will inevitably hit the wall when $p$ is even larger than $n$, owing to the fact that $\log|\hat{\Sigma}|$ is not well-defined in this situation.  In addition, in \cite{JBZ2013}, another test statistic constructed from the traces of F-matrices, the so-called {\emph{trace criterion test}}, was proposed. Under $\mathbf{H}_0$, a CLT was derived for this statistic under the following restrictions
\begin{eqnarray}
\frac{p_i}{p_1+\cdots+p_{i-1}}\to r_1^{(i)}\in(0,\infty),\quad \frac{p_i}{n-1-(p_1+\cdots+p_{i-1})}\to r_2^{(i)}\in(0,1), \label{052207}
\end{eqnarray}
for all $i\in\llbracket k\rrbracket$, together with $p-p_1<n$. We stress here, condition (\ref{052207}) is obviously much stronger than $p_i/n\to y_i\in (0,1)$ for all $i\in\llbracket k \rrbracket$.

Roughly speaking, our aim in this paper is to raise a new statistic with both statistical visualizability and mathematical tractability, whose limiting behavior can be derived with the following restriction on the dimensionality and the sample size
\begin{eqnarray}
p_i:=p_i(n),\quad p_i/n\to y_i\in(0,1),\quad i\in\llbracket k\rrbracket,\quad \text{ as } n\to \infty. \label{3192}
\end{eqnarray}
Especially, $n$ is not required to be larger than $p$ or any partial sum of $p_i$. More precisely, our multi-facet target consists of the following:
\begin{itemize}
\item Introducing a new matrix model tailored to ({\bf{T1}}), namely {\emph{block correlation matrix}}, which can be viewed as a natural high-dimensional extension of the sample correlation matrix.
\item Constructing the so-called {\emph{Schott type statistic}} from the block correlation matrix, which can be regarded as an extension of Schott's statistic for complete independence test in \cite{S2005}.
\item Deriving the limiting distribution of the Schott type statistic with the aid of tools from the Free Probability Theory ({\bf{FPT}}). Specifically, we will channel the so-called {\emph{real second order freeness}}  for Haar distributed orthogonal matrices from \cite{MP2013} into the framework.
\item Employing this limiting law to test independence of $k$ sub-vectors under (\ref{3192}) and  assessing the statistic via simulations and a real data set,  which  comes from the daily returns of 258 stocks issued by the companies from S\&P 500.
\end{itemize}
It will be seen that for ({\bf{T1}}), it is quite natural and reasonable to put forward the concept of block correlation matrix. Just like that the sample correlation matrix is designed to bypass the unknown mean values and variances of entries, the block correlation matrix can also be employed, without knowing the population mean vectors $\boldsymbol{\mu}_i$ and covariance matrices $\Sigma_{ii}$, $i\in\llbracket k\rrbracket$. Also interestingly, it turns out that, the statistical meaning of the Schott type statistic is rooted in the idea of testing independence based on the Canonical Correlation Analysis ({\bf{CCA}}). Meanwhile, theoretically, the Schott type statistic is a special type of linear spectral statistic from the perspective of {\bf{RMT}}. Then methodologies from {\bf{RMT}} and {\bf{FPT}} can take over the derivation of the limiting behavior of the proposed statistic.
As far as we know, the application of {\bf{FPT}} in high-dimensional statistical inference is still in its infancy. We also hope this work can evoke more applications of {\bf{FPT}} in statistical inference in the future. One may refer to \cite{RMSE2008} for another application of {\bf{FPT}}, in the context of statistics.

Our paper is organized as follows. In Section 2, we will construct the block correlation matrix. Then we will present the definition of the Schott type statistic and its limiting law in Section 3. Meanwhile, we will discuss the statistical rationality of this test statistic, especially the relationship with {\bf{CCA}}. In Section 4, we will detect the utility of our statistic by simulation, and an example about stock prices will be analyzed in Section 5. Finally, Section 6 will be devoted to the illustration of how {\bf{RMT}} and {\bf{FPT}} machinery can help to establish the limiting behavior of our statistic, and some calculations will be presented in the Appendix.

Throughout the paper, $\text{tr}\mathbf{A}$ represents the trace of a square matrix $\mathbf{A}$. If $\mathbf{A}$ is $N\times N$, we will use $\lambda_1(\mathbf{A}),\ldots,\lambda_N(\mathbf{A})$ to denote its $N$ eigenvalues. Moreover, $|\mathbf{A}|$ means the determinant of $\mathbf{A}$. In addition, $\mathbf{0}_{M\times N}$ will be used to denote the $M\times N$ null matrix, and be abbreviated to $\mathbf{0}$ when there is no confusion on the dimension. Moreover, for random objects $\xi$ and $\eta$, we use $\xi\stackrel{d}=\eta$ to represent that $\xi$ and $\eta$ share the same distribution.
\section{Block correlation matrix}
An elementary property of the sample correlation matrix is that it is invariant under translation and scaling of variables. Such an advantage allows us to discuss the limiting behavior of statistics constructed from the sample correlation matrix without knowing the means and variances of the involved variables, thus makes it a favorable choice in dealing with the real data. Now we are in a similar situation, without knowing explicit information of the population mean vectors $\boldsymbol{\mu}_i$ and covariance matrices $\Sigma_{ii}$, we want to construct a test statistic which is independent of these unknown parameters, in a similar vein. The first step is to propose a high-dimensional extension of sample correlation matrix, namely the block correlation matrix. For simplicity, we use the notation
\begin{eqnarray*}
\text{diag}(\mathbf{A}_i)_{i=1}^k=\left(
\begin{array}{ccc}
\mathbf{A}_1 &~ &~\\
~ &\ddots &~\\
~ &~ &\mathbf{A}_k
\end{array}
\right)
\end{eqnarray*}
to denote the diagonal block matrix with blocks $\mathbf{A}_i, i\in\llbracket k \rrbracket$, i.e. all off-diagonal blocks are $\mathbf{0}$.
\begin{defn}[(Block correlation matrix)]
With the aid of the notation in (\ref{3193}), the block correlation matrix $\mathbf{B}:=\mathbf{B}(\mathbf{X}_1,\cdots,\mathbf{X}_k)$ is defined as follows
\begin{eqnarray*}
\mathbf{B}&:=&\text{diag}((\mathbf{X}_i\mathbf{X}'_i)^{-\frac12})_{i=1}^k\cdot \mathbf{X}\mathbf{X}'\cdot \text{diag}((\mathbf{X}_i\mathbf{X}'_i)^{-\frac12})_{i=1}^k\nonumber\\\\
&=& \left((\mathbf{X}_i\mathbf{X}'_i)^{-1/2}\mathbf{X}_i\mathbf{X}'_j(\mathbf{X}_j\mathbf{X}'_j)^{-1/2}\right)_{i,j=1}^k.
\end{eqnarray*}
\end{defn}
\begin{rmk} Note that when $p_i=1$ for all $i\in\llbracket k\rrbracket$, $\mathbf{B}$ is reduced to the classical sample correlation matrix of $n$ observations of a $k$-dimensional random vector. In this sense, we can regard $\mathbf{B}$ as a natural high-dimensional extension of the sample correlation matrix.
\end{rmk}
\begin{rmk} If we take determinant of $\mathbf{B}$, we can get the likelihood ratio statistic. However, since one needs to further take logarithm on the determinant, the assumption $n>p+2$ is indispensable, in light of \cite{JY2013}.
\end{rmk}
In the sequel, we perform a very standard and well known transformation for $\mathbf{X}$ to eliminate the inconvenience caused by subtracting the sample mean. Set the orthogonal matrix
\begin{eqnarray*}
A=\left(
\begin{array}{cccccc}
\frac{1}{\sqrt{n}} &\frac{1}{\sqrt{n}} &\frac{1}{\sqrt{n}} &\cdots &\frac{1}{\sqrt{n}}\\
\frac{1}{\sqrt{2}} &-\frac{1}{\sqrt{2}} &0 &\cdots &0\\
\frac{1}{\sqrt{3\cdot2}} &\frac{1}{\sqrt{3\cdot 2}} &-\frac{2}{\sqrt{3\cdot2}} &\cdots &0\\
\cdots &\cdots &\cdots &\cdots &\cdots\\
\frac{1}{\sqrt{n(n-1)}} &\frac{1}{\sqrt{n(n-1)}} &\frac{1}{\sqrt{n(n-1)}} &\cdots &-\frac{n-1}{\sqrt{n(n-1)}}
\end{array}
\right).
\end{eqnarray*}
Then we can find that there exist i.i.d. $\mathbf{z}(j)\sim N(0,\Sigma),j\in\llbracket n-1\rrbracket$, such that
\begin{eqnarray*}
\mathbf{X}A':=(\mathbf{0}, \mathbf{z}(1),\cdots, \mathbf{z}(n-1) ),
\end{eqnarray*}
Analogously, we denote
\begin{eqnarray*}
\mathbf{X}_iA':=(\mathbf{0}, \mathbf{z}_i(1),\cdots, \mathbf{z}_i(n-1) ),\quad i\in\llbracket k\rrbracket.
\end{eqnarray*}
Obviously, $\mathbf{z}(j)=(\mathbf{z}_1'(j),\cdots, \mathbf{z}_k'(j))'$.
To abbreviate, we further set the matrices
\begin{eqnarray*}
\mathbf{Z}=(\mathbf{z}(1),\cdots, \mathbf{z}(n-1)),\quad \mathbf{Z}_i=(\mathbf{z}_i(1),\cdots, \mathbf{z}_i(n-1)),\quad i\in\llbracket k\rrbracket.
\end{eqnarray*}
Apparently, we have $\mathbf{Z}\mathbf{Z}'=\mathbf{X}\mathbf{X}'$ and $\mathbf{Z}_i\mathbf{Z}_i'=\mathbf{X}_i\mathbf{X}'_i$. Consequently, we can also write
\begin{eqnarray*}
\mathbf{B}
&=&\text{diag}(\mathbf{Z}_i\mathbf{Z}_i')^{-\frac12})_{i=1}^k\cdot \mathbf{Z}\mathbf{Z}' \cdot
\text{diag}(\mathbf{Z}_i\mathbf{Z}_i')^{-\frac12})_{i=1}^k\nonumber\\\\
&=& \left((\mathbf{Z}_i\mathbf{Z}_i')^{-1/2}\mathbf{Z}_i\mathbf{Z}_j'(\mathbf{Z}_j\mathbf{Z}_j')^{-1/2}\right)_{i,j=1}^k.
\end{eqnarray*}
An advantage of $\mathbf{Z}_i$ over $\mathbf{X}_i$ is that its entries are i.i.d.

\section{Schott type statistic and main result}
With the block correlation matrix at hand, we can propose our test statistic for ($\mathbf{T1}$), namely the Schott type statistic. Such a nomenclature is motivated by \cite{S2005} on another classical independence test problem, the so-called {\emph{complete independence test}}, which can be described as follows. Given a random vector $\mathbf{w}=(w_1,\ldots, w_p)'$, we consider the following hypothesis testing:
\begin{eqnarray*}
(\mathbf{T2})\hspace{5ex}\tilde{\mathbf{H}}_0: w_1,\cdots, w_p \text{ are completely independent} \hspace{5ex} \text{v.s.} \hspace{5ex} \tilde{\mathbf{H}}_1: \text{ not } \tilde{\mathbf{H}}_0.
\end{eqnarray*}
When $\mathbf{w}$ is multivariate normal, the above test problem is equivalent to the so-called test of sphericity of population correlation matrix, i.e.  under the null hypothesis, the population correlation matrix of $\mathbf{w}$ is $\mathbf{I}_p$. There is a long list of references devoted to testing complete independence for a random vector under the high-dimensional setting. See, for example, \cite{Johnstone2001}, \cite{LW2002}, \cite{Srivastava05S}, \cite{CJ2011},  \cite{BJYZ2009} and \cite{S2005}. Especially, in \cite{S2005}, the author constructed a statistic from the sample correlation matrix. To be specific, denote the sample correlation matrix of $n$ i.i.d. observations of $\mathbf{w}$ by $\mathbf{R}=\mathbf{R}(n,p):=(r_{ij})_{p\times p}$. Schott's statistic for ($\mathbf{T2}$) is then defined as follows
\begin{eqnarray*}
\boldsymbol{s}(\mathbf{R}):=\sum_{i=2}^p\sum_{j=1}^{i-1} r_{ij}^2=\frac{1}{2}(\sum_{i,j=1}^pr_{ij}^2-p)=\frac{1}{2}\text{tr}\mathbf{R}^2-\frac{p}{2}.
\end{eqnarray*}
Note that under the null hypothesis, all off-diagonal entries of the population correlation matrix should be $0$. Hence, it is quite natural to use the summation of squares of all off-diagonal entries to measure the difference between the population correlation matrix and $I$. Then Schott's statistic $\mathbf{s}(\mathbf{R})$ is just the sample counterpart of such a measurement.
 Now in a similar vein, we define the {\emph{Schott type statistic}} for the block correlation matrix as follows.
\begin{defn}[(Schott type statistic)] We define the Schott type statistic of the block correlation matrix $\mathbf{B}$ by
\begin{eqnarray*}
\boldsymbol{s}(\mathbf{B}):=\frac{1}{2}\text{tr}\mathbf{B}^2-\frac{p}{2}=\frac12\sum_{\ell=1}^p\lambda_\ell^2(\mathbf{B})-\frac{p}{2}.
\end{eqnarray*}
\end{defn}
For simplicity, we introduce the matrix
\begin{eqnarray*}
\mathbf{C}(i,j):=(\mathbf{X}_i\mathbf{X}'_i)^{-1/2}\mathbf{X}_i\mathbf{X}'_j(\mathbf{X}_j\mathbf{X}'_j)^{-1}\mathbf{X}_j\mathbf{X}_i'(\mathbf{X}_i\mathbf{X}'_i)^{-1/2}.
\end{eqnarray*}
Stemming from the definition of $\mathbf{B}$, we can easily get
\begin{eqnarray*}
\boldsymbol{s}(\mathbf{B})=\frac12\sum_{i,j=1}^k \text{tr}\mathbf{C}(i,j)-\frac{p}{2}=\sum_{i=2}^p\sum_{j=1}^{i-1} \text{tr} \mathbf{C}(i,j).
\end{eqnarray*}
The matrix $\mathbf{C}(i,j)$ is frequently used in {\bf{CCA}} towards random vectors $\mathbf{x}_i$ and $\mathbf{x}_j$, known as the canonical correlation matrix. A well known fact is that the eigenvalues of $\mathbf{C}(i,j)$ provide meaningful measures of the correlation between these two random vectors. Actually, its singular values (square root of eigenvalues) are the well-known canonical correlations between $\mathbf{x}_i$ and $\mathbf{x}_j$. Thus the summation of all eigenvalues, $\text{tr} \mathbf{C}(i,j)$, also known as Pillai's test statistic, can obviously measure the correlation between $\mathbf{x}_i$ and $\mathbf{x}_j$. Summing up these measurements over all $(i,j)$ pairs (subject to $j<i$) yields our Schott type statistic, which can capture the overall correlation among $k$ parts of $\mathbf{x}$. Thus from the perspective of {\bf{CCA}}, the Schott type statistic possesses a theoretical validity for the testing problem ({\bf{T1}}).
%

Our main result is the following CLT under $\mathbf{H}_0$.
\begin{thm}\label{th1}Assume that $p_i:=p_i(n)$ and $p_i/n\to y_i\in(0,1)$ for all $i\in\llbracket k\rrbracket$, we have
\begin{eqnarray*}
\frac{\boldsymbol{s}(\mathbf{B})-a_n}{\sqrt{b_n}}\Longrightarrow N(0,1), \quad \text{as}\quad n\to\infty,
\end{eqnarray*}
where
\begin{eqnarray*}
a_n=\frac{1}{2}\sum_{\substack{i,j,   i\neq j}} \frac{p_ip_j}{n-1},\qquad  b_n=\sum_{\substack{i,j,   i\neq j}} \frac{p_ip_j(n-1-p_i)(n-1-p_j)}{(n-1)^4}.
\end{eqnarray*}
\end{thm}
\begin{rmk} Note that by assumption, $a_n$ is of order $O(n)$ and $b_n$ is of order $O(1)$.
\end{rmk}

\section{Numerical studies}
In this section we compare the performance of the statistics proposed in \cite{JY2013}, \cite{JBZ2013} and our Schott type statistic, under various settings of sample size and dimensionality. For simplicity, we will focus on the case of $k=3$.  The Wilks' statistic in (\ref{3191}) without any renormalization has been shown to perform badly for ($\mathbf{T1}$)
in \cite{JY2013} and \cite{JBZ2013}, thus will not be considered in this section. Under the null hypothesis $\mathbf{H}_0$, the samples are drawn from the following two distributions:

\vspace{0.2cm}
(\uppercase\expandafter{\romannumeral1}) $\boldsymbol{\mu}_i=\mathbf{0}_{p_i\times1}$, $\Sigma_{ii}=\mathbf{I}_{p_i}$;

(\uppercase\expandafter{\romannumeral2}) $\boldsymbol{\mu}_i=(\mu_{i1},\dots,\mu_{ip_i})'$, $\Sigma_{ii}=\text{diag}(d_{i1},\dots,d_{ip_i})$, where $\mu_{ij}\thicksim  U(-1,1)$ and $d_{ij}\thicksim \chi^2_8$. \vspace{0.2cm}\newline
Here $U(-1,1)$ denotes the uniform distribution with the support $(-1,1)$ and $\chi^2_8$ denotes the chi-squared distribution  with eight degrees of freedom.
Under the alternative hypothesis $\mathbf{H}_1$, we adopt two distributions introduced in \cite{JY2013} and \cite{JBZ2013} respectively:

\vspace{0.2cm}
(\uppercase\expandafter{\romannumeral3}) $\boldsymbol{\mu}=\mathbf{0}_{p\times1}$, $\Sigma=0.15\mathbf{1}_{p\times p}+0.85\mathbf{I}_p$;

(\uppercase\expandafter{\romannumeral4}) $\boldsymbol{\mu}_i=\mathbf{0}_{p_i\times1}$, $\Sigma_{ii}=26/25\mathbf{I}_{p_i}$, $\Sigma_{12}=1/25\mathds{I}_{12}$, $\Sigma_{13}=6/25\mathds{I}_{13}$ and $\Sigma_{23}=6/25\mathds{I}_{23}$. \vspace{0.2cm}\newline
Here $\mathbf{1}_{p\times p}$ stands for the matrix whose entries are all equal to 1 and  $\mathds{I}_{ij}$ stands for the rectangular matrix whose main diagonal entries are  equal to 1 and the others are  equal to 0.
The  empirical sizes and  powers are obtained based on 100,000 replications.
In the tables,
$T_0$ denotes the proposed Schott type test, $T_1$ denotes the renormalized likelihood ratio test of \cite{JY2013} and $T_2$ denotes the trace criterion test of  \cite{JBZ2013}.

Table 1 reports the empirical sizes of three tests at the $5\%$ significance level for scenario \uppercase\expandafter{\romannumeral1} and  scenario  \uppercase\expandafter{\romannumeral2}. It shows that the proposed Schott type test $T_0$ performs quite robust with respect to these two distributions. Even when $p_1=2,p_2=2,p_3=3$ and $n=4$ the attained significance levels are also less than $8\%$. It is obvious that the empirical size of $T_0$ is good enough in the  inapplicable cases of $T_1$ and $T_2$. Furthermore,  in addition to these inapplicable cases,  the test $T_0$ also performs better than $T_1$ in terms  of size and has similar empirical sizes with $T_2$.

The empirical powers of three tests  for scenarios \uppercase\expandafter{\romannumeral3} and \uppercase\expandafter{\romannumeral4} are presented in Table 2. The most gratifying thing is that $T_0$ performs much better than $T_1$ and $T_2$ in both scenarios.  It is worthy to notice that when the sample size $n$ is large, the renormalized likelihood ratio test $T_1$  is also satisfactory. But the  performance of $T_2$ is not so good.

\begin{table}[h]
  \centering

 {\small
\begin{tabular}{ccccc|ccc}
    \hline
    \multirow{2}{*}{$(p_1,p_2,p_3)$} & \multirow{2}{*}{$n$}&\multicolumn{3}{c}{Scenario \uppercase\expandafter{\romannumeral1} }&\multicolumn{3}{c}{Scenario \uppercase\expandafter{\romannumeral2} }\\ \cline{3-8}
&&$T_{0}$   &$T_{1}$&$T_{2}$&$T_{0}$&$T_{1}$&$T_{2}$\\ \hline
$(2,2,3)$   &$ 4 $ & 0.07796 &   N.A.  &   N.A.  & 0.07747 &   N.A.  &   N.A.  \\
            &$ 6 $ & 0.04787 &   N.A.  & 0.04305 & 0.04963 &   N.A.  & 0.04439 \\
            &$10 $ & 0.05314 & 0.07448 & 0.04577 & 0.05296 & 0.07430 & 0.04696 \\
            &$16 $ & 0.05659 & 0.07222 & 0.05340 & 0.05685 & 0.06968 & 0.05301 \\
            &$30 $ & 0.06027 & 0.06815 & 0.05931 & 0.06118 & 0.06806 & 0.05943 \\
            &$50 $ & 0.06070 & 0.06496 & 0.06016 & 0.06309 & 0.06667 & 0.06187 \\
$(10,10,15)$&$20 $ & 0.04824 &   N.A.  &   N.A.  & 0.04900 &   N.A.  &   N.A.  \\
            &$30 $ & 0.04942 &   N.A.  & 0.04852 & 0.04874 &   N.A.  & 0.04767 \\
            &$40 $ & 0.04873 & 0.05998 & 0.04771 & 0.05030 & 0.06072 & 0.04881 \\
            &$50 $ & 0.05117 & 0.05781 & 0.05034 & 0.05019 & 0.05649 & 0.04857 \\
            &$100$ & 0.05257 & 0.05473 & 0.05113 & 0.05186 & 0.05556 & 0.05236 \\
            &$150$ & 0.05257 & 0.05445 & 0.05232 & 0.05240 & 0.05448 & 0.05201 \\
$(30,30,45)$&$60 $ & 0.04978 &   N.A.  &   N.A.  & 0.04856 &   N.A.  &   N.A.  \\
            &$90 $ & 0.04924 &   N.A.  & 0.04914 & 0.04975 &   N.A.  & 0.04949 \\
            &$110$ & 0.04913 & 0.05653 & 0.04868 & 0.05059 & 0.05668 & 0.04948 \\
            &$130$ & 0.04982 & 0.05293 & 0.04999 & 0.05067 & 0.05396 & 0.05057 \\
            &$150$ & 0.05034 & 0.05210 & 0.04963 & 0.05244 & 0.05323 & 0.05152 \\
            &$180$ & 0.04993 & 0.05196 & 0.04924 & 0.05084 & 0.05215 & 0.04936 \\
$(50,50,75)$&$100$ & 0.04980 &   N.A.  &   N.A.  & 0.04847 &   N.A.  &   N.A.  \\
            &$150$ & 0.04963 &   N.A.  & 0.05001 & 0.04946 &   N.A.  & 0.04926 \\
            &$180$ & 0.05182 & 0.05554 & 0.05126 & 0.04970 & 0.05383 & 0.04971 \\
            &$210$ & 0.04897 & 0.05229 & 0.04966 & 0.05041 & 0.05386 & 0.05021 \\
            &$250$ & 0.04838 & 0.05059 & 0.04913 & 0.04970 & 0.05126 & 0.04958 \\
            &$300$ & 0.05033 & 0.05133 & 0.04991 & 0.05002 & 0.05103 & 0.05037 \\
\hline

\end{tabular}}
 \caption{Empirical sizes of  tests $T_0$, $T_1$ and $T_2$ at  the $5\%$ significance level for scenario \uppercase\expandafter{\romannumeral1} and  scenario  \uppercase\expandafter{\romannumeral2} }\label{tb1}
\end{table}

\begin{table}[h]
  \centering
 {
\begin{tabular}{ccccc|ccc}
    \hline
    \multirow{2}{*}{$(p_1,p_2,p_3)$} & \multirow{2}{*}{$n$}&\multicolumn{3}{c}{Scenario \uppercase\expandafter{\romannumeral3} }&\multicolumn{3}{c}{Scenario \uppercase\expandafter{\romannumeral4} }\\ \cline{3-8}
&&$T_{0}$   &$T_{1}$&$T_{2}$&$T_{0}$&$T_{1}$&$T_{2}$\\ \hline
$(2,2,3)$   &$ 4 $ & 0.08219 &   N.A.  &   N.A.  & 0.07720 &   N.A.  &   N.A.  \\
            &$ 6 $ & 0.06070 &   N.A.  & 0.05122 & 0.05868 &   N.A.  & 0.04795 \\
            &$10 $ & 0.09727 & 0.10294 & 0.07617 & 0.08600 & 0.10186 & 0.06695 \\
            &$16 $ & 0.15997 & 0.15084 & 0.12610 & 0.13575 & 0.14768 & 0.11424 \\
            &$30 $ & 0.32599 & 0.28965 & 0.26751 & 0.27091 & 0.27946 & 0.24789 \\
            &$50 $ & 0.55628 & 0.50614 & 0.48825 & 0.49040 & 0.49669 & 0.47230 \\
$(10,10,15)$&$20 $ & 0.09288 &   N.A.  &   N.A.  & 0.07375 &   N.A.  &   N.A.  \\
            &$30 $ & 0.19754 &   N.A.  & 0.13685 & 0.14449 &   N.A.  & 0.09356 \\
            &$40 $ & 0.35574 & 0.22503 & 0.22320 & 0.25350 & 0.18131 & 0.17422 \\
            &$50 $ & 0.54114 & 0.42118 & 0.34155 & 0.38352 & 0.31902 & 0.28675 \\
            &$100$ & 0.98802 & 0.97426 & 0.90124 & 0.92197 & 0.90747 & 0.88459 \\
            &$150$ & 0.99997 & 0.99992 & 0.99828 & 0.99813 & 0.99792 & 0.99713 \\
$(30,30,45)$&$60 $ & 0.15294 &   N.A.  &   N.A.  & 0.17557 &   N.A.  &   N.A.  \\
            &$90 $ & 0.40723 &   N.A.  & 0.24709 & 0.59527 &   N.A.  & 0.30862 \\
            &$110$ & 0.61546 & 0.39853 & 0.37262 & 0.84084 & 0.46911 & 0.60497 \\
            &$130$ & 0.79664 & 0.76933 & 0.51235 & 0.95858 & 0.85025 & 0.84310 \\
            &$150$ & 0.91131 & 0.93370 & 0.65023 & 0.99305 & 0.97156 & 0.96005 \\
            &$180$ & 0.98457 & 0.99439 & 0.82161 & 0.99974 & 0.99890 & 0.99765 \\
$(50,50,75)$&$100$ & 0.17700 &   N.A.  &   N.A.  & 0.32881 &   N.A.  &   N.A.  \\
            &$150$ & 0.48045 &   N.A.  & 0.28762 & 0.93635 &   N.A.  & 0.62012 \\
            &$180$ & 0.68780 & 0.49147 & 0.42095 & 0.99589 & 0.76337 & 0.92511 \\
            &$210$ & 0.85001 & 0.89883 & 0.56308 & 0.99991 & 0.99495 & 0.99500 \\
            &$250$ & 0.95905 & 0.99340 & 0.73447 & 1.00000 & 0.99999 & 0.99997 \\
            &$300$ & 0.99567 & 0.99993 & 0.88908 & 1.00000 & 1.00000 & 1.00000 \\
\hline

\end{tabular}}
  \caption{Empirical powers of  tests $T_0$, $T_1$ and $T_2$ at the   $5\%$ significance level for scenario \uppercase\expandafter{\romannumeral3} and  scenario  \uppercase\expandafter{\romannumeral4} }
\end{table}

\section{An example}
For illustration, we apply the proposed test statistic to the daily returns of  258 stocks issued by the  companies  from S\&P 500. The original data are the  closing prices or the  bid/ask average of these stocks for the trading days of the last quarter in 2013, i.e., from 1 October 2013 to 31 December 2013, with total 64 days. This  dataset is derived from  the Center for Research in Security Prices  Daily Stock in Wharton Research Data Services. According to The North American Industry Classification System (NAICS),  which  is used by business and government to classify business establishments,  the 258 stocks are separated into 11  sectors. Numbers of stocks  in each sector are shown in Table \ref{tb3}.  A common  interest here is to test whether the daily returns for the investigated 11 sectors are independent.
\begin{table}[h]
  \centering
\begin{tabular}{c|c|c|c|c|c|c|c|c|c|c|c}
    \hline
Sector&\ \ 1\ \ &\ \ 2\ \ &\ \ 3\ \  &\ \ 4\ \ &\ \ 5\ \ &\ \ 6\ \ &\ \ 7\ \ &\ \ 8\ \ &\ \ 9\ \ &\ \ 10\ \ &\ \ 11\ \ \\\hline
Number of stocks&30&32&14&32&12&33&55&14&16&10&10\\\hline
\end{tabular}
\caption{Number of stocks in each NAICS Sectors. Sector 1 describes mining, quarrying, and oil and gas extraction; Sector 2 describes utilities;  Sector 3 describes wholesale trade; Sector 4 describes retail trade; Sector 5 describes transportation and warehousing; Sector 6 describes information; Sector 7 describes finance and insurance; Sector 8 describes real estate and rental and leasing; Sector 9 describes professional, scientific, and technical services; Sector 10 describes administrative and support and waste management and remediation services; Sector 11 describes arts, entertainment, and recreation.  }\label{tb3}
\end{table}

The testing model is  established  as  follows:  Denote $p_i$ as the number of stocks  in the
$i$th sector, $u_{il}(j)$ as the price for the $l$th stock  in the
$i$th sector at day $j$. Here $j\in\llbracket 64\rrbracket$. Correspondingly we have $p=\sum_{i=1}^{11}p_i=258$.  In order to   satisfy the condition of  the proposed test statistic, the original data $u_{il}(j)$ need to be  transformed  as follows:  (\romannumeral1) {\emph{{logarithmic difference}}:  Let  $x_{il}(j)=\ln(u_{il}(j+1)/u_{il}(j))$.
Notice that  $j\in\llbracket 63\rrbracket$. So we denote $n=63$.
Logarithmic difference  is a very commonly used procedure in finance. There are a number of theoretical and practical advantages of using logarithmic returns, especially we shall assume that the sequence of logarithmic returns are independent of each other for  big  time scales (e.g.  $\geq $ 1 day, see \cite{Rama(01)}). (\romannumeral2) {\emph{{power transform}}: It is  well known that if a security price follows geometric Brownian motion,  then the logarithmic returns of the security are normally distributed. However in	 most	 cases, the normalized logarithmic returns $x_{il}(j)$ are  considered to have sharper peaks and  heavier tails than the standard normal distribution. Thus we first  transform $x_{il}(j)$ to $\hat x_{il}(j)$ by Box-Cox transformation, and then suppose the  transformed
 data follows a standard normal distribution, that  is
\begin{align*}
\tilde x_{il}(j)=\left(\frac{\hat x_{il}(j)-\bar x_{il}}{\hat\sigma_{il}}\right)^{\beta_{il}}\sim N(0,1).
\end{align*}
 Here  $\beta_{il}$ is an unknown parameter, $\bar x_{il}$ and $\hat\sigma_{il} $ are the sample mean and sample standard deviation of  $\hat x_{il}(j), j\in\llbracket n\rrbracket$. $\beta_{il}$ can be estimated by
 \begin{align*}
\frac{1}{n}\sum_{j=1}^n\tilde x_{il}^4(j)\approx \int_{a_{il}}^{b_{il}}t^4d\Phi(t),
\end{align*}
where  ${a_{il}}=\min_{j\in\llbracket n\rrbracket}\tilde x_{il}(j)$, ${b_{il}}=\max_{j\in\llbracket n\rrbracket}\tilde x_{il}(j)$ and $\Phi(t)$ is the standard normal distribution function.
(\romannumeral3) {\emph{{normality test }}:  we use the Kolmogorov-Smirnov test to test whether  the  transformed prices of each stock are drawn from the  standard normal distribution. By calculation, we get 258 $p$-values.  Among them the minimum is 0.1473. And 91.86\% of $p$-values are bigger than 0.5. For illustration, we present  the smoothed empirical densities of the transformed data  $\tilde x_{il}(j)$ for the first four stocks of each sector in Figure \ref{fig3}.
\begin{figure}[h]
\centering
\includegraphics[scale=0.85]{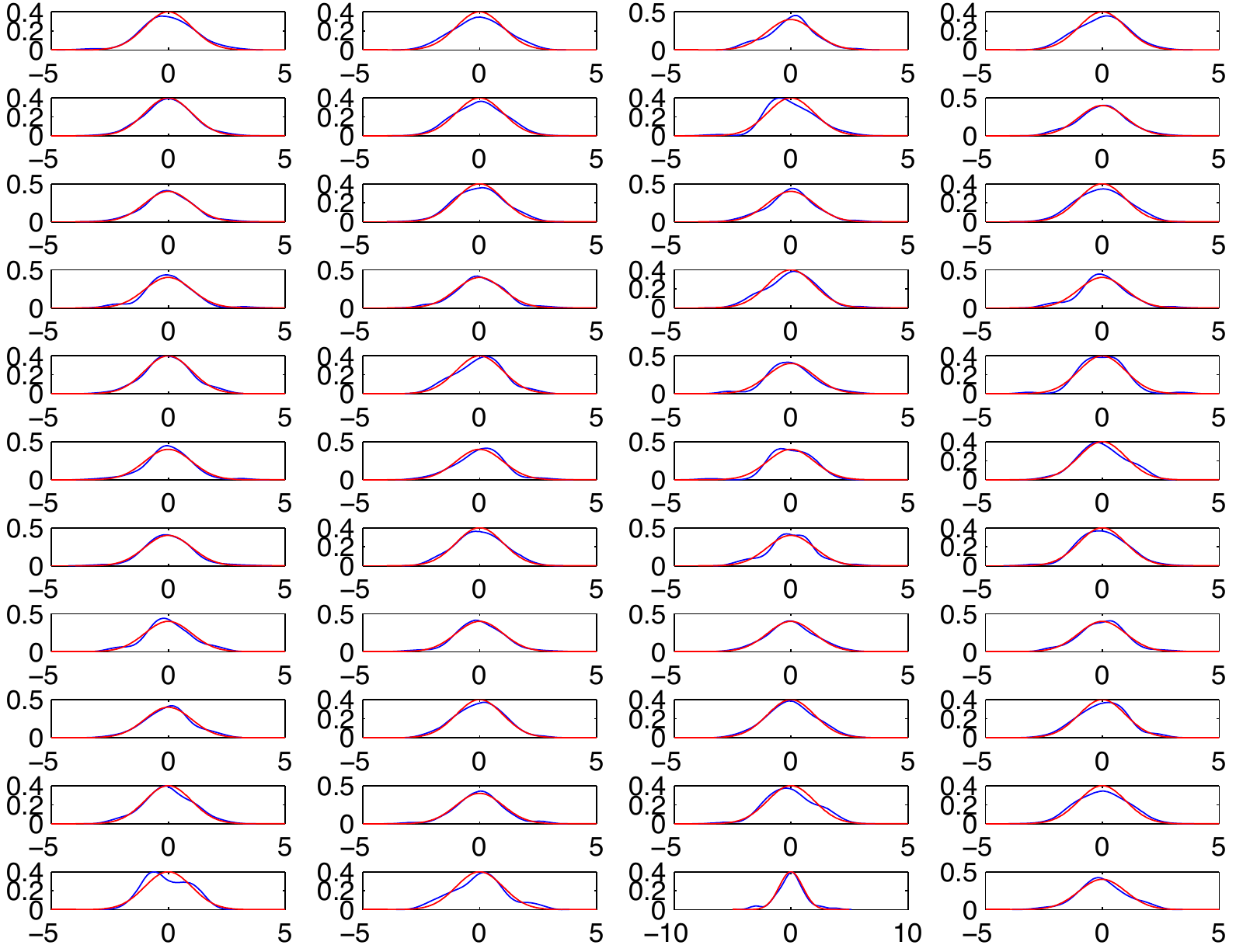}
\caption{Graphs of the empirical density function of the transformed data versus the standard normal distribution.
These graphs contain the empirical density functions of the transformed data for  the first four stocks of each sector used in our study. The blue curve is the smoothed density function of the transformed data for one stock and the red curve is standard normal density function.}\label{fig3}
\end{figure}
From these graphs, we can also see that the transformed data fit the normal density curve well.
Thus from the above arguments,  we  can assume that the transformed data $\tilde x_{il}(j)$ satisfy the conditions in Theorem \ref{th1}.

Now we apply $\boldsymbol{s}(\mathbf{B})$ to test the independence of every two sectors.
The $p$-values are shown in Table \ref{tb4}.  We find in the total 55 pairs of sectors,  there are 23 pairs with $p$-values bigger than 0.05  and 18 pairs with $p$-values bigger than 0.1. Interestingly,  according to these results, if we set the significance level as 5\% we find that there are  seven sectors which are independent of
Sector 7, the finance and insurance sector, which seems to be most independent of other sectors. On the other hand,  Sector 11 (the arts, entertainment, and recreation sector)  is  dependent on all other sectors except Sector 1 (the mining, quarrying, and oil and gas extraction sector). We also investigate the mutual independence of  every three sectors.
Applying Theorem \ref{th1},  we find that there are 11 groups with $p$-values bigger than 0.05. In addtion, we find that there is only one group containing four sectors which are mutually independent.  The    results   are shown in Table \ref{tb5}. Thus we have strong evidence to believe that every five sectors are dependent.
\begin{table}[h]
  \centering{
\begin{tabular}{c|cccccccccc}

Sector&1 & 2 &3 &4&5& 6&7 &8 &9&10 \\\hline
2&0.0405&&&&&&&&&\\
3&0.0487&0.2735&&&&&&&&\\
4&0.0604&0         &0.0002&&&&&&&\\
5&0.0012&0.1041&0.0027&0.1639&&&&&&\\
6&0.0073&0.0048&0.0008&0.4285&0.0444&&&&&\\
7&0.2299&0.4830&0.0451&0.1053&0.7080&0.6843&&&&\\
8&0.3558&0.0208&0.2458&0.3547&0.0013&0.0645&0.1127&&&\\
9&0.1411&0.0833&0         &0.0005&0.2403&0.0124&0.4521&0.0048&&\\
10&0.0418&0.3075&0.0004&0.0847&0.0746&0.0026&0.0036&0.0109&0&\\
11&0.8689&0.0048&0         &0.0470&0.0167&0.0004&0.0490&0         &0.0003&0.0003\\
\end{tabular}
\caption{The $p$-values obtained by the proposed test  under  $\mathbf{H}_0$ with $k=2$. Notice that the results are  rounded up to the fourth decimal point. }\label{tb4}
}
\end{table}

\begin{table}[h]
  \centering{

\begin{tabular}{c|ccccccccccc}
\hline
Sectors&(1,4,8)&(1,7,8)&(1,7,9)&(2,5,7) &(2,7,9)&(4,5,7)\\\hline
$p$-value&0.0804&0.1135&0.1120&0.2541&0.1413&0.1263\\\hline\hline
Sectors&(4,6,7)&(4,6,8)&(4,7,8)&(5,7,9)&(6,7,8)&(4,6,7,8)\\\hline
$p$-value&0.3088&0.1277&0.0686&0.3913&0.0855&0.0650\\\hline
\end{tabular}
\caption{The $p$-values obtained by the proposed test  under  $\mathbf{H}_0$ with $k=3$ and $k=4$. Notice that the results are  rounded up to the fourth decimal point. }\label{tb5}
}
\end{table}

\section{Linear spectral statistics and second order freeness}
In this section, we will introduce some {\bf{RMT}} and {\bf{FPT}} apparatus with which Theorem \ref{th1} can be proved.  We will just summarize the main ideas on how these tools fit into the framework of the limiting behavior of our proposed statistic, but leave the details of the reasoning and calculation to the Appendix. We start from the following elementary fact. To wit,
for two matrices $\mathbf{S}$ and $\mathbf{T}$, we know that $\mathbf{S}\mathbf{T}$ and $\mathbf{T}\mathbf{S}$ share the same non-zero eigenvalues, as long as both $\mathbf{S}\mathbf{T}$ and $\mathbf{T}\mathbf{S}$ are square. Therefore, to study the eigenvalues of $\mathbf{B}$, it is equivalent to study the eigenvalues of the $(n-1)\times (n-1)$ matrix
\begin{eqnarray*}
\mathbf{\underline{B}}:=\mathbf{Z}'\text{diag}[(\mathbf{Z}_i\mathbf{Z}'_i)^{-1}]_{i=1}^k\mathbf{Z}=\sum_{i=1}^k\mathbf{Z}'_i(\mathbf{Z}_i\mathbf{Z}'_i)^{-1}\mathbf{Z}_i.
\end{eqnarray*}
Setting
\begin{eqnarray*}
\boldsymbol{s}(\mathbf{\underline{B}})=\frac{1}{2}\sum_{\ell=1}^{n-1} \lambda_\ell^2(\mathbf{\underline{B}})-\frac{p}{2},
\end{eqnarray*}
by the above discussion we can assert
\begin{eqnarray*}
\boldsymbol{s}(\mathbf{\underline{B}})=\boldsymbol{s}(\mathbf{B}).
\end{eqnarray*}
A main advantage of $\mathbf{\underline{B}}$ is embodied in the following proposition, which will be a starting point of the proof of Theorem \ref{th1}.
\begin{prop} \label{prop.1}Assume that $\mathbf{O}_1,\ldots, \mathbf{O}_k$ are i.i.d. $(n-1)$-dimensional random orthogonal matrices possessing Haar distribution on the orthogonal group $O(n-1)$. Let $\mathbf{P}_i=\mathbf{I}_{p_i} \oplus \mathbf{0}_{n-1-p_i}$ be a diagonal projection matrix with rank $p_i$, for each $i\in\llbracket k\rrbracket$. Here we denote by $\mathbf{0}_m$ the $m\times m$ null matrix. Then under $\mathbf{H}_0$, we have
\begin{eqnarray*}
\underline{\mathbf{B}}\stackrel{d}=\sum_{i=1}^k \mathbf{O}_i'\mathbf{P}_i\mathbf{O}_i.
\end{eqnarray*}
\end{prop}
\begin{proof} We perform the singular value decomposition as $\mathbf{Z}_i:=\mathbf{U}_i'\mathbf{D}_i \mathbf{O}_i$, where $\mathbf{U}_i$ and $\mathbf{O}_i$ are $p_i$-dimensional and $(n-1)$-dimensional orthogonal matrices respectively, and $\mathbf{D}_i$ is a $p_i\times (n-1)$ rectangular matrix whose main diagonal entries are nonzero while the others are all zero. It is well known that when $\mathbf{Z}_i$ has i.i.d. normal entries, both $\mathbf{U}_i$ and $\mathbf{O}_i$ are Haar distributed. Then an elementary calculation leads to Proposition 1.
\end{proof}
Note that Proposition \ref{prop.1} allows us to study the eigenvalues of a summation of $k$ independent random projections instead. Moreover, it is obvious that under $\mathbf{H}_0$, this summation of random matrices does not depend on the unknown population mean vectors and  covariance matrices of $\mathbf{x}_i$, $i\in\llbracket k\rrbracket$.

To compress notation, we set
\begin{eqnarray}
\mathbf{Q}:=\sum_{i=1}^k \mathbf{O}_i'\mathbf{P}_i\mathbf{O}_i,\quad \mathbf{Q}\stackrel{d}=\underline{\mathbf{B}}. \label{3251}
\end{eqnarray}
According to the discussion in the last section, we know that our Schott type statistic can be expressed (in distribution) in terms of the eigenvalues of $\mathbf{Q}$, since
\begin{eqnarray*}
\text{tr}\mathbf{B}^2\stackrel{d}=\text{tr} \mathbf{Q}^2=\sum_{\ell=1}^{n-1}\lambda_{\ell}^2(\mathbf{Q}).
\end{eqnarray*}
In {\bf{RMT}}, given an $N\times N$ random matrix $\mathbf{A}$ and some test function $f: \mathbb{C}\to \mathbb{C}$, one usually calls the quantity
\begin{eqnarray*}
\mathcal{L}_N[f]:=\sum_{\ell=1}^Nf(\lambda_\ell(\mathbf{A}))
\end{eqnarray*}
a {\emph{linear spectral statistic}} of $\mathbf{A}$ with test function $f$. For some classical random matrix models such as Wigner matrices and sample covariance matrices, linear spectral statistics have been widely studied. Not trying to be comprehensive, one can refer to \cite{BS2004}, \cite{Johansson1998}, \cite{LP2009}, \cite{SS1998} and \cite{Shcherbina2011} for instance. A notable feature in this type of CLTs is that usually the variance of the linear spectral statistic is of order $O(1)$ when the test function $f$ is smooth enough, mainly due to the strong correlations among eigenvalues, thus is significantly different from the case of i.i.d variables. Now, in a similar vein, with the random matrix $\mathbf{Q}$ at hand, we want to study the fluctuation of its linear spectral statistics,  focusing on the test function $f(x)=x^2$.

In the past few decades,  the main stream of {\bf{RMT}} has focused on the spectral behavior of single random matrix models such as Wigner matrix, sample covariance matrix and non-Hermitian matrix with i.i.d. variables. However, with the rapid development in {\bf{RMT}} and its related fields, the study of general polynomials with classical single random matrices as its variables is in increasing demand.
A favorable idea is to derive the spectral properties of the matrix polynomial from the information of the spectrums of its variables (single matrices).
Specifically, the question can be described as \\

{\emph{Given the eigenvalues of $\mathbf{A}$ and $\mathbf{B}$, what can one say about the eigenvalues of $h(\mathbf{A},\mathbf{B})$?}}\\\\
Here $h(\cdot,\cdot)$ is a bivariate polynomial. Usually, for deterministic matrices $\mathbf{A}$ and $\mathbf{B}$, only with their eigenvalues given, it is impossible to write down the eigenvalues of $h(\mathbf{A},\mathbf{B})$. However, for some  independent high-dimensional random matrices $\mathbf{A}$ and $\mathbf{B}$, deriving the limiting spectral properties of $h(\mathbf{A},\mathbf{B})$ via those
of $\mathbf{A}$ and $\mathbf{B}$ is possible.  To answer this kind of question, the right machinery to employ is {\bf{FTP}}. In the breakthrough work by \cite{V1991}, the author proved that if $(\mathbf{A}_n)_{n\in\mathbb{N}}$ and $(\mathbf{B}_n)_{n\in\mathbb{N}}$ are two independent sequences of  random Hermitian matrices  and at least one of them is orthogonally invariant (in distribution), then they satisfy the property of {\emph {asymptotic freeness}}, which allows one to derive the limit of $\frac{1}{n}\mathbb{E}\text{tr} h(\mathbf{A}_n,\mathbf{B}_n)$ from the limits of $(\frac{1}{n}\mathbb{E}\text{tr}\mathbf{A}_n^m)_{m\in\mathbb{N}}$ and $(\frac{1}{n}\mathbb{E}\text{tr}\mathbf{B}_n^m)_{m\in\mathbb{N}}$ directly. Sometimes we also call the asymptotic freeness of two random matrix sequences as {\emph{first order freeness}}.

Our aim in this paper, however, is not to derive the limit of the normalized trace of some polynomial in random matrices,
but to take a step further to study the fluctuation of the trace. To this end, we need to adopt the concept of {\emph{second order freeness}}, which was recently raised and developed in the series of work:  \cite{MS2006, MSS2007, CMSS2007, MP2013}, also see \cite{Redelmeier2013}. In contrast, the second order freeness aims at answering how to derive the fluctuation property of $\text{tr} h(\mathbf{A}_n, \mathbf{B}_n)$ from the limiting spectral properties of $\mathbf{A}_n$ and $\mathbf{B}_n$.

Especially, in \cite{MP2013}, the authors established the so-called {\emph{real second order freeness}} for orthogonal matrices, which is specialized in solving the fluctuation of the linear spectral statistics of polynomials in Haar distributed orthogonal matrices and deterministic matrices. For our purpose, we need to employ Proposition 52 in \cite{MP2013}, an ad hoc and simplified version of which can be heuristically sketched as follows. Assume that  $\{\mathbf{A}_n\}_{n\in\mathbb{N}}$ and $\{\mathbf{B}_n\}_{n\in\mathbb{N}}$ are two independent sequences of random matrices (may be deterministic), where $\mathbf{A}_n$ and $\mathbf{B}_n$ are $n$ by $n$, and the limits of $n^{-1}\mathbb{E}\text{tr} h_1(\mathbf{A}_n)$ and $n^{-1}\mathbb{E}\text{tr} h_2(\mathbf{B}_n)$ exist for any given polynomials $h_1$ and $h_2$, as $n\to\infty$. Moreover, $\text{tr} h_1(\mathbf{A}_n)$ and $\text{tr} h_2(\mathbf{B}_n)$ possess Gaussian fluctuations (may be degenerate) asymptotically for any given polynomials $h_1$ and $h_2$. Then $\text{tr} q (\mathbf{O}'\mathbf{A}_n\mathbf{O}, \mathbf{B}_n)$ also possesses Gaussian fluctuation asymptotically for any given bivariate polynomial $q(\cdot,\cdot)$, where $\mathbf{O}$ is supposed to be an $n\times n$ Haar orthogonal matrix independent of $\mathbf{A}_n$ and $\mathbf{B}_n$.

Now as for $\mathbf{Q}$, we can start from the case of $k=2$, which fits the above framework quite well. To wit, we can regard $\mathbf{O}_1'\mathbf{P}_1\mathbf{O}_1$ as $\mathbf{B}_{n-1}$ and $\mathbf{P}_2$ as $\mathbf{A}_{n-1}$, using Proposition 52 of \cite{MP2013} leads to the fact that $\text{tr} h(\mathbf{O}_1'\mathbf{P}_1\mathbf{O}_1+\mathbf{O}_2'\mathbf{P}_2\mathbf{O}_2)$ is asymptotically Gaussian after  an appropriate normalization, for any given polynomial $h$. Then we take $\mathbf{O}_1'\mathbf{P}_1\mathbf{O}_1+\mathbf{O}_2'\mathbf{P}_2\mathbf{O}_2$ as $\mathbf{B}_{n-1}$ and regard $\mathbf{P}_3$ as $\mathbf{A}_{n-1}$ and repeat the above discussion. By using Proposition 52 in \cite{MP2013} recursively, we can get our CLT for $\mathbf{Q}$ finally. A formal result which can be derived from Proposition 52 in \cite{MP2013} is as follows, whose proof will be presented in the Appendix.
\begin{thm} \label{th2}  For our matrix $\mathbf{Q}$ defined in (\ref{3251}), and any deterministic polynomial sequence $h_1,h_2,$ $h_3,\cdots$, we have
\begin{eqnarray*}
\text{Cov}(\text{tr } h_1(\mathbf{Q}), \text{tr } h_2(\mathbf{Q}))=O(1)
\end{eqnarray*}
and
\begin{eqnarray*}
\lim_{n\to\infty} \kappa_r(\text{tr } h_1(\mathbf{Q}),\ldots, \text{tr } h_r(\mathbf{Q}))=0,\quad \text{if } r\geq 3
\end{eqnarray*}
where $\kappa_r(\xi_1,\ldots,\xi_r)$ represents the joint cumulant of the random variables $\xi_1,\ldots,\xi_r$.
\end{thm}
Now if we set $h_i(x)=x^2$ for all $i\in\mathbb{N}$, we can obtain Theorem \ref{th1} by proving the following lemma.
\begin{lem} \label{lem1} Under the above notation, we have
\begin{eqnarray}
\mathbb{E}\text{tr} \mathbf{Q}^2=p+\sum_{\substack{i,j,  i\neq j}} \frac{p_ip_j}{n-1}, \label{0524001}
\end{eqnarray}
and
\begin{eqnarray}
\text{Var}(\text{tr}\mathbf{Q}^2)=4\sum_{\substack{i,j,  i\neq j}} \frac{p_ip_j(n-1-p_i)(n-1-p_j)}{(n-1)^4}+O(\frac{1}{n}). \label{0524002}
\end{eqnarray}
\end{lem}
The proof of Lemma \ref{lem1} will also be stated in the Appendix. In the sequel, we prove Theorem \ref{th1} with Lemma \ref{lem1} granted.

\noindent
{\bf Proof of Theorem \ref{th1}}.
Setting $h_i(x)=x^2$ for all $i\in\mathbb{N}$ in Theorem \ref{th2}, we see that all $r$th cumulants of $\text{tr}\mathbf{Q}^2$ tend to $0$ if $r\geq 3$ when $n\to \infty$, which together with Lemma \ref{lem1} implies that
\begin{eqnarray*}
\frac{\text{tr} \mathbf{Q}^2-\mathbb{E}\text{tr} \mathbf{Q}^2}{\sqrt{\text{Var} (\text{tr} \mathbf{Q}^2)}}\Longrightarrow N(0,1).
\end{eqnarray*}
Then Theorem \ref{th1} follows from the definition of $\mathbf{s}(\mathbf{B})$
 directly. \qed

\section*{Appendix}
In this Appendix, we provide the proofs of Theorem \ref{th2} and Lemma \ref{lem1}.

\noindent
{\bf Proof of Theorem \ref{th2}}. In Proposition 52 of \cite{MP2013}, the authors state that if $\{\mathbf{A}_n\}_{n\in\mathbb{N}}$ and $\{\mathbf{B}_n\}_{n\in\mathbb{N}}$ are two independent sequences of random matrices (may be deterministic), where $\mathbf{A}_n$ and $\mathbf{B}_n$ are $n\times n$,  each having a {\emph{real second order limit distribution}}, and $\mathbf{O}$ is supposed to be an $n\times n$ Haar orthogonal matrix independent of $\mathbf{A}_n$ and $\mathbf{B}_n$, then $\mathbf{B}_n$ and $\mathbf{O}'\mathbf{A}_n\mathbf{O}$ are {\emph{asymptotically real second order free}}.  By Definition 30 of \cite{MP2013}, we see that for a single random matrix sequence $\{\mathbf{A}_n\}_{n\in\mathbb{N}}$, the existence of the so-called {\emph{real second order limit distribution}} means that the following three statements hold simultaneously for any given sequence of polynomials $h_1,h_2,h_3,\ldots$ when $n\to\infty$:
\begin{enumerate}
\item[1)] $n^{-1}\mathbb{E}\text{tr} h_1(\mathbf{A}_n)$ converges;
 \item [2)] $\text{Cov}(\text{tr} h_1(\mathbf{A}_n), \text{tr} h_2(\mathbf{A}_n))$ converges (the limit can be $0$);
\item [3)] $\kappa_r(\text{tr} h_1(\mathbf{A}_n),\ldots, \text{tr} h_r(\mathbf{A}_n))=o(1)$ for all $r\geq 3$.
 \end{enumerate}
We stress here, the original definition in \cite{MP2013} is given with a language of non-commutative probability theory. To avoid introducing too many additional notions, we just modify it to be the above 1)-3).
Then by the Definition 33 and Proposition 52 of \cite{MP2013}), one see that if both $\{\mathbf{A}_n\}_{n\in\mathbb{N}}$ and $\{\mathbf{B}_n\}_{n\in\mathbb{N}}$ have real second order limit distributions, we have the following three facts for any given sequences of bivariate polynomials $q_1,q_2,q_3,\ldots$ when  $n\to\infty$:
 \begin{enumerate}
 \item[1')] $n^{-1}\mathbb{E}\text{tr} q_1(\mathbf{B}_n, \mathbf{O}'\mathbf{A}_n\mathbf{O})$ converges;
 \item[2')] $\text{Cov}(\text{tr} q_1(\mathbf{B}_n, \mathbf{O}'\mathbf{A}_n\mathbf{O}), \text{tr} q_2(\mathbf{B}_n, \mathbf{O}'\mathbf{A}_n\mathbf{O}))$ converges;
 \item[3')] $\kappa_r(\text{tr} q_1(\mathbf{B}_n,\mathbf{O}'\mathbf{A}_n\mathbf{O}),\ldots, \text{tr} q_r(\mathbf{B}_n, \mathbf{O}'\mathbf{A}_n\mathbf{O}))=o(1)$  for all $r\geq 3$.
 \end{enumerate}
Here 1') and 2') can be implied by the definitions of the first and second order freeness in \cite{MP2013} respectively, and 3') can be found in the proof of Proposition 52 of \cite{MP2013}, where the authors claim that the proof of Theorem 41 therein is also applicable under the setting of Proposition 52. Note that in \cite{MP2013}, a more concrete rule to determine the limit of  $\text{Cov}(\text{tr} q_1(\mathbf{B}_n, \mathbf{O}'\mathbf{A}_n\mathbf{O}), \text{tr} q_2(\mathbf{B}_n, \mathbf{O}'\mathbf{A}_n\mathbf{O}))$ is given, which can be viewed as the core of the concept of {\emph{real second order freeness}}. However, here we do not need such a concrete rule, thus do not introduce it.

Now we are at the stage of employing Proposition 52 of \cite{MP2013} to prove Theorem \ref{th2}. Recall our objective $\mathbf{Q}$ defined in (\ref{3251}). We start from the case of $k=2$, to wit, we are considering the linear spectral statistics of the random matrix $\mathbf{O}_1'\mathbf{P}_1\mathbf{O}_1+\mathbf{O}_2'\mathbf{P}_2\mathbf{O}_2$. Now, we regard $\mathbf{O}_1'\mathbf{P}_1\mathbf{O}_1$ as $\mathbf{B}_{n-1}$ and $\mathbf{P}_2$ as $\mathbf{A}_{n-1}$, then obviously they both satisfy 1)-3) in the definition of  the existence of the {\emph{real second order limit distribution}}, since the spectrums of $\mathbf{A}_{n-1}$ and $\mathbf{B}_{n-1}$ are both deterministic, noticing they are projection matrices with known ranks. Then 1')-3') immediately imply that $\mathbf{O}_1'\mathbf{P}_1\mathbf{O}_1+\mathbf{O}_2'\mathbf{P}_2\mathbf{O}_2$ also has a real second order limit distribution. Next, adding $\mathbf{O}_3'\mathbf{P}_3\mathbf{O}_3$ to $\mathbf{O}_1'\mathbf{P}_1\mathbf{O}_1+\mathbf{O}_2'\mathbf{P}_2\mathbf{O}_2$, and regarding the latter as $\mathbf{B}_{n-1}$ and $\mathbf{P}_3$ as $\mathbf{A}_{n-1}$, we can use the above discussion again to conclude that $\mathbf{O}_1'\mathbf{P}_1\mathbf{O}_1+\mathbf{O}_2'\mathbf{P}_2\mathbf{O}_2+\mathbf{O}_3'\mathbf{P}_3\mathbf{O}_3$ also possesses a real second order limit distribution. Recursively, we can finally get that $\mathbf{Q}$ has a real second order limit distribution, which implies Theorem \ref{th2}. So we conclude the proof. \qed

It remains to prove Lemma \ref{lem1}. Before commencing the proof, we briefly introduce some technical inputs. Since the trace of a product of matrices can always be expressed in terms of some products of their entries, it is expected that we will need to calculate the quantities of the form
\begin{eqnarray}
\mathbb{E}\mathbf{O}_{i_1j_1}\cdots\mathbf{O}_{i_mj_m},  \label{052701}
\end{eqnarray}
where $\mathbf{O}$ is assumed to be an $N$-dimensional Haar distributed orthogonal matrix, and $\mathbf{O}_{ij}$ is its $(i,j)$th entry. A powerful tool handling this kind of expectation is the so-called {\emph{Weingarten calculus on orthogonal group}}, we refer to the seminal paper of \cite{CS2006}, formula (21) therein. To avoid introducing too many combinatorics notions for Weingarten calculus, we just list some consequence of it for our purpose, taking into account the fact that we will only need to handle the case of $m\leq 4$ in (\ref{052701}) in the sequel. Specifically, we have the following lemma.

\begin{lem}\label{lem2} Under the above notation, we have the following facts for (\ref{052701}), assuming $m\leq 4$ and $\mathbf{i}=(i_1,\ldots,i_m)$ and $\mathbf{j}=(j_1,\ldots,j_m)$.\\\\
1): When $m=2$, $i_1=i_2$, and $j_1=j_2$, we have  $(\ref{052701})=N^{-1}$,\\\\
2): When $m=4$, we have the following results for four sub-cases.
\begin{enumerate}
\item[i):] If $i_1=i_2=i_3=i_4$, $j_1=j_2=j_3=j_4$, we have $(\ref{052701})=3/(N(N+2))$;\\
\item[ii):] If $i_1=i_2=i_3=i_4$, $j_1=j_2\neq j_3=j_4$, we have $(\ref{052701})=1/(N(N+2))$;\\
\item[iii):] If $i_1=i_2\neq i_3=i_4$, $j_1=j_2\neq j_3=j_4$, we have $(\ref{052701})=(N+1)/(N(N-1)(N+2))$;\\
\item[iv):] If $i_1=i_3\neq i_2=i_4$, $j_1=j_2\neq j_3=j_4$, we have $(\ref{052701})=-1/(N(N-1)(N+2))$.
\end{enumerate}
3): Replacing $\mathbf{O}$ by $\mathbf{O}'$, we can obviously switch the roles of $\mathbf{i}$ and $\mathbf{j}$ in 1) and 2). Moreover, any permutation on the indices $\{1,\ldots,m\}$ will not change (\ref{052701}). Any other triple $(m,\mathbf{i},\mathbf{j})$, which can not be transformed  into any case in 1) or 2) via switching the roles of $\mathbf{i}$ and $\mathbf{j}$
 or performing permutations on the indices $\{1,\ldots,m\}$, will drives (\ref{052701}) to be $0$.
 \end{lem}
 With Lemma \ref{lem2} at hand, we can prove Lemma \ref{lem1} in the sequel.

 \noindent
{\bf Proof of Lemma \ref{lem1}}.
At first, we verify (\ref{0524001}). Note that by definition,
\begin{eqnarray*}
\mathbb{E}\text{tr} \mathbf{Q}^2&=&\sum_{i=1}^k\mathbb{E}\text{tr}(\mathbf{O}_i'\mathbf{P}_i\mathbf{O}_i)^2+\sum_{\substack{i,j,i\neq j}}\mathbb{E}\text{tr}(\mathbf{O}_i'\mathbf{P}_i\mathbf{O}_i\cdot \mathbf{O}_j'\mathbf{P}_j\mathbf{O}_j) \nonumber\\
&=& p+\sum_{\substack{i,j,  i\neq j}}\mathbb{E}\text{tr}(\mathbf{O}_i'\mathbf{P}_i\mathbf{O}_i\cdot \mathbf{O}_j'\mathbf{P}_j\mathbf{O}_j).
\end{eqnarray*}
Let $\mathbf{u}_i(\ell)$ be the $\ell$th column of $\mathbf{O}_i'$ and $\mathbf{u}_i(\ell,s)$ be the $s$th coefficient of $\mathbf{u}_i(\ell)$, i.e. the $(\ell,s)$th entry of $\mathbf{O}_i$, for all $i\in\llbracket k \rrbracket$. Then for $i\neq j$ we have
\begin{eqnarray}
\mathbb{E}\text{tr}(\mathbf{O}_i'\mathbf{P}_i\mathbf{O}_i\cdot \mathbf{O}_j'\mathbf{P}_j\mathbf{O}_j)&=&\mathbb{E}\sum_{m=1}^{p_i}\sum_{\ell=1}^{p_j}\text{tr} \mathbf{u}_i(m)\mathbf{u}_i'(m)\mathbf{u}_j(\ell)\mathbf{u}_j'(\ell)\nonumber\\
&=& \sum_{m=1}^{p_i}\sum_{\ell=1}^{p_j}\sum_{s}\sum_{t} \mathbb{E}\mathbf{u}_i(m,s)\mathbf{u}_i(m,t)\cdot\mathbb{E}\mathbf{u}_j(\ell,s)\mathbf{u}_j(\ell,t)\nonumber\\
&=&  \sum_{m=1}^{p_i}\sum_{\ell=1}^{p_j}\sum_{s} \mathbb{E}(\mathbf{u}_i(m,s))^2\mathbb{E}(\mathbf{u}_j(\ell,s))^2=\frac{p_ip_j}{n-1}. \label{052702}
\end{eqnarray}
Here, in the third step above we used 3) of Lemma \ref{lem2} to discard the terms with $s\neq t$, while in the last step we used 1) of Lemma \ref{lem2}.
Therefore, we have
\begin{eqnarray}
\mathbb{E}\text{tr} (\sum_{i=1}^k\mathbf{O}_i'\mathbf{P}_i\mathbf{O}_i)^2=p+\sum_{\substack{i,j,i\neq j}} \frac{p_ip_j}{n-1},\label{mean}
\end{eqnarray}
Now we calculate $\text{Var} (\text{tr} \mathbf{Q}^2)$ as follows. Note that we have
\begin{eqnarray*}
\text{Var}(\text{tr} \mathbf{Q}^2)&=&\mathbb{E}(\text{tr} (\sum_{i=1}^k\mathbf{O}_i'\mathbf{P}_i\mathbf{O}_i)^2)^2-(\mathbb{E}\text{tr}(\sum_{i=1}^k\mathbf{O}_i'\mathbf{P}_i\mathbf{O}_i)^2)^2\nonumber\\
&=& \mathbb{E}\bigg(\sum_{\substack{i,j,  i\neq j}}\text{tr}(\mathbf{O}_i'\mathbf{P}_i\mathbf{O}_i\cdot \mathbf{O}_j'\mathbf{P}_j\mathbf{O}_j) \bigg)^2-\bigg(\mathbb{E}\sum_{\substack{i,j,  i\neq j}}\text{tr}(\mathbf{O}_i'\mathbf{P}_i\mathbf{O}_i\cdot \mathbf{O}_j'\mathbf{P}_j\mathbf{O}_j)\bigg)^2\nonumber\\
&=&\sum_{\substack{i,j,  i\neq j}} \sum_{\substack{m,\ell,  m\neq \ell}} \text{Cov}(\text{tr}(\mathbf{O}_i'\mathbf{P}_i\mathbf{O}_i\cdot \mathbf{O}_j'\mathbf{P}_j\mathbf{O}_j), \text{tr}(\mathbf{O}_m'\mathbf{P}_m\mathbf{O}_m\cdot \mathbf{O}_\ell'\mathbf{P}_\ell\mathbf{O}_\ell))\nonumber\\
&=& \sum_{\substack{i,j,m,\ell\\  i\neq j,m\neq \ell\\ \{i,j\}\cap\{m,\ell\}\neq \emptyset}} \text{Cov}(\text{tr}(\mathbf{O}_i'\mathbf{P}_i\mathbf{O}_i\cdot \mathbf{O}_j'\mathbf{P}_j\mathbf{O}_j), \text{tr}(\mathbf{O}_m'\mathbf{P}_m\mathbf{O}_m\cdot \mathbf{O}_\ell'\mathbf{P}_\ell\mathbf{O}_\ell)).
\end{eqnarray*}
In the sequel, we briefly write
\begin{eqnarray*}
\text{Cov}((i,j),(m,\ell)):=\text{Cov}(\text{tr}(\mathbf{O}_i'\mathbf{P}_i\mathbf{O}_i\cdot \mathbf{O}_j'\mathbf{P}_j\mathbf{O}_j), \text{tr}(\mathbf{O}_m'\mathbf{P}_m\mathbf{O}_m\cdot \mathbf{O}_\ell'\mathbf{P}_\ell\mathbf{O}_\ell))
\end{eqnarray*}
Note that the summation in the last step above can be decomposed into the following six cases.
\begin{eqnarray*}
&&1: m=i,\quad \ell=j;\qquad \hspace{2ex}2: m=j,\quad \ell=i;\nonumber\\
&&3: m=i,\quad \ell\neq i, j;\qquad 4: \ell=i,\quad m\neq i,j;\nonumber\\
&& 5: m=j,\quad \ell\neq i,j;\qquad 6: \ell=j,\quad m\neq i,j.
\end{eqnarray*}
Now given $i,j$, we decompose the summation over $m,\ell$ according to the above $6$ cases and denote the sum restricted on these cases by $\sum_{\alpha(i,j)},\alpha(i,j)=1,\ldots,6$ respectively. Therefore,
\begin{eqnarray}
\text{Var}(\text{tr} \mathbf{Q}^2)=\sum_{\substack{i,j,  i\neq j}} \sum_{\alpha(i,j)=1}^6 \sum\nolimits_{\alpha(i,j)} \text{Cov}((i,j),(m,\ell)) \label{052705}
\end{eqnarray}
Now by definition we have
\begin{eqnarray*}
&&\sum_{\substack{i,j,  i\neq j}} \sum\nolimits_{\alpha(i,j)} \text{Cov}((i,j),(m,\ell))=\sum_{i}\sum_{j\neq i} \text{Cov}((i,j),(i,j)),\qquad \hspace{1ex}\alpha=1,2,\nonumber\\
&&\sum_{\substack{i,j,  i\neq j}} \sum\nolimits_{\alpha(i,j)}\text{Cov}((i,j),(m,\ell))=\sum_{i}\sum_{j\neq i}\sum_{\ell\neq i,j} \text{Cov}((i,j),(i,\ell)),\quad \alpha=3,4,\nonumber\\
&&\sum_{\substack{i,j,  i\neq j}} \sum\nolimits_{\alpha(i,j)} \text{Cov}((i,j),(m,\ell))=\sum_{i}\sum_{j\neq i}\sum_{\ell\neq i,j} \text{Cov}((i,j),(j,\ell)), \quad\alpha=5,6.\nonumber\\
\end{eqnarray*}
Now note that for $i\neq j$
\begin{eqnarray*}
\text{Cov}((i,j),(i,j))&=&\mathbb{E}(\text{tr}(\mathbf{O}_i'\mathbf{P}_i\mathbf{O}_i\cdot \mathbf{O}_j'\mathbf{P}_j\mathbf{O}_j))^2-(\mathbb{E}\text{tr}(\mathbf{O}_i'\mathbf{P}_i\mathbf{O}_i\cdot \mathbf{O}_j'\mathbf{P}_j\mathbf{O}_j))^2\nonumber\\
&=&\mathbb{E}(\text{tr}(\mathbf{O}_i'\mathbf{P}_i\mathbf{O}_i\cdot \mathbf{O}_j'\mathbf{P}_j\mathbf{O}_j))^2-(\frac{p_ip_j}{n-1})^2,
\end{eqnarray*}
where the last step follows from (\ref{052702}).
Moreover, we have
\begin{eqnarray*}
&&\mathbb{E}(\text{tr}(\mathbf{O}_i'\mathbf{P}_i\mathbf{O}_i\cdot \mathbf{O}_j'\mathbf{P}_j\mathbf{O}_j))^2\nonumber\\
&&=\sum_{m_1,m_2=1}^{p_i}\sum_{\ell_1,l_2=1}^{p_j}\mathbb{E}\mathbf{u}_i'(m_1)\mathbf{u}_j(\ell_1)\mathbf{u}_j'(\ell_1)\mathbf{u}_i(m_1)\mathbf{u}_i'(m_2)\mathbf{u}_j(\ell_2)\mathbf{u}_j'(\ell_2)\mathbf{u}_i(m_2)\nonumber\\
&&=\sum_{s_1,s_2,t_1,t_2}\bigg{[}\sum_{m_1,m_2=1}^{p_i}\mathbb{E}\mathbf{u}_i(m_1,s_1)\mathbf{u}_i(m_1,t_1)\mathbf{u}_i(m_2,s_2)\mathbf{u}_i(m_2,t_2)\bigg{]}  \nonumber\\
&&\hspace{15ex}\times\bigg{[}\sum_{\ell_1,\ell_2=1}^{p_j}\mathbb{E}\mathbf{u}_j(\ell_1,s_1)\mathbf{u}_j(\ell_1,t_1)\mathbf{u}_j(\ell_2,s_2)\mathbf{u}_j(\ell_2,t_2)\bigg{]}.
\end{eqnarray*}
To calculate the above expectation, we need to use Lemma \ref{lem2} again. In light of 3) of Lemma \ref{lem2},  it suffices to consider the following four cases
\begin{eqnarray*}
&&1: s_1=s_2=t_1=t_2,\quad 2: s_1=s_2\neq t_1=t_2\nonumber\\
&&3: s_1=t_1\neq s_2=t_2,\quad 4: s_1=t_2\neq s_2=t_1.
\end{eqnarray*}
Through detailed but elementary calculation, with the aid of Lemma \ref{lem2}, we can finally obtain that for $i\neq j$,
\begin{eqnarray*}
\text{Cov}((i,j),(i,j))=\frac{2p_ip_j(n-1-p_i)(n-1-p_j)}{(n-1)^4}+O(\frac{1}{n}).
\end{eqnarray*}
Moreover, analogously, when $i,j,\ell$ are mutually distinct, we can get
\begin{eqnarray*}
\text{Cov}((i,j),(i,\ell))= \text{Cov}((i,j),(j,\ell))=O(\frac{1}{n})
\end{eqnarray*}
by using Lemma \ref{lem2}. Here we just omit the details of the calculation.
Consequently, by (\ref{052705}) one has
\begin{eqnarray}
\text{Var}(\text{tr} \mathbf{Q}^2)=\sum_{\substack{i,j,  i\neq j}}\frac{4p_ip_j(n-1-p_i)(n-1-p_j)}{(n-1)^4}+O(\frac{1}{n}). \label{variance}
\end{eqnarray}
Thus we conclude the proof. \qed


\end{document}